\documentclass{amsart}
\usepackage{amsfonts}
\usepackage{amsmath,amssymb}
\usepackage{amsthm}
\usepackage{amscd}
\usepackage{graphics}
\usepackage{graphicx}

\theoremstyle{remark}{
\newtheorem{Def}{{\rm Definition}}
\newtheorem{Ex}{{\rm Example}}
\newtheorem{Rem}{{\rm Remark}}

\newtheorem*{MainProb}{Main Problem}
}
\theoremstyle{plain}
{

\newtheorem{Prop}{Proposition}

\newtheorem{MainThm}{Main Theorem}

\newtheorem{Lem}{Lemma}

}

\begin{document}
\title[Reeb graphs of smooth functions on $3$-dimensional closed manifolds]{On Reeb graphs induced from smooth functions on $3$-dimensional closed manifolds which may not be orientable}
\author{Naoki Kitazawa}
\keywords{Smooth functions and maps. Reeb spaces and Reeb graphs. Morse functions and fold maps, Differential topology.\\
\indent {\it \textup{2020} Mathematics Subject Classification}: Primary~57R45. Secondary~57R19.}
\address{Institute of Mathematics for Industry, Kyushu University, 744 Motooka, Nishi-ku Fukuoka 819-0395, Japan\\
 TEL (Office): +81-92-802-4402 \\
 FAX (Office): +81-92-802-4405 \\
}
\email{n-kitazawa@imi.kyushu-u.ac.jp}
\urladdr{https://naokikitazawa.github.io/NaokiKitazawa.html}
\maketitle
\begin{abstract}
The {\it Reeb space} of a smooth function is
a topological and combinatoric object and fundamental and important in understanding topological and geometric properties of the manifold of the domain. It is the graph and a topological space endowed with a natural topology. This is defined as the quotient space of the manifold of the domain where the equivalence relation is as follows: two points in the manifold are equivalent if and only if they are in a same connected component of a level set or a preimage. In considerable cases they are graphs ({\it Reeb graphs}): if the function is a so-called {\it Morse}(-{\it Bott}) functions for example, then this is the graph such that a point is a vertex if and only if the corresponding connected component of the level set contains some singular points.

The author previously constructed explicit smooth functions on suitable $3$-dimensional connected, closed and orientable manifolds whose Reeb graphs are isomorphic to prescribed graphs and whose preimages are as prescribed types. This gives a new answer to so-called realization problems of graphs as Reeb graphs of smooth functions of suitable classes. The present paper concerns a variant in the case where the $3$-dimensional manifolds may not be non-orientable extending the result before.       
\end{abstract}

\section{Introduction}
\label{sec:1}

The {\it Reeb space} of a smooth function $c$ is defined as follows.

For a smooth map $c:X \rightarrow Y$, we can define an equivalence relation ${\sim}_c$ on $X$ as follows: $x_1 {\sim}_{c} x_2$ holds if and only if they are in a same connected component of a preimage $c^{-1}(y)$.

\begin{Def}
The quotient space $W_c:=X/{\sim}_c$ is the {\it Reeb space} of $c$.
\end{Def}
Hereafter, $q_c:X \rightarrow W_c$ denotes the quotient space. We can define a map which is denoted by $\bar{c}$ uniquely by the relation $c=\bar{c} \circ q_c$. 

For a smooth manifold $X$, $T_pX$ denotes the tangent space at $p$.
A {\it singular} point of a smooth map $c:X \rightarrow Y$ is a point $p \in X$ where the rank of the differential $dc_p:T_pX \rightarrow T_c(p)Y$ is smaller than $\min \{\dim X,\dim Y\}$. The {\it singular set} of $c$ is the set of all singular points.
For a smooth map $c$, the {\it singular value} is a point $c(p)$ which is a value at some singular point $p$. A {\it regular value} of the map is a point which is not a singular value in $Y$. 

Reeb spaces are in considerable cases graphs where the vertex set is the set of all points $p$ the preimages ${q_f}^{-1}(p)$ of which contain at least one singular point of $f$. 
Reeb spaces are Reeb graphs for smooth functions on compact manifolds with finitely many singular values for example (\cite{saeki4}). In the present paper we only concentrate on such smooth functions essentially.

One of pioneering paper on Reeb spaces is \cite{reeb} for example and they have been fundamental and important topological objects and tools in algebraic topological studies and differential topological ones on differentiable manifolds. 

They inherit topological information such as homology groups and cohomology rings. See \cite{kitazawa0.1}--\cite{kitazawa0.6} of the author and \cite{saekisuzuoka} for example for related expositions. 
They essentially concentrate on {\it fold} maps such that preimages of regular values are disjoint unions of spheres. {\it Fold} maps are higher dimensional variants of so-called Morse functions and the definition of a fold map is introduced in the next section.

The present paper studies another important problem as follows.
\begin{MainProb}
For a finite and connected graph with at least one edge, can we construct a smooth function on a (compact) manifold (satisfying some good conditions) inducing a Reeb graph isomorphic to the graph? 
\end{MainProb}

\cite{sharko} is a pioneering paper on this. \cite{martinezalfaromezasarmientooliveira}, \cite{masumotosaeki}, \cite{michalak} and \cite{saeki4} are some of  related studies among the others. The author also obtained results in \cite{kitazawa} and \cite{kitazawa2}.
Recently Reeb spaces are also important in applied or applications of mathematics such as data analysis and visualizations as \cite{sakuraisaekicarrwuyamamotoduketakahashi} shows for example. This problem will play important roles in such scenes.

Note that a graph is naturally homeomorphic to a $1$-dimensional polyhedron.

An {\it isomorphism} between two graphs $G_1$ and $G_2$ is a homeomorphism from $G_1$ to $G_2$ mapping the vertex set of $G_1$ onto the vertex set of $G_2$. 

A graph has a continuous function in the following if and only if it has no loop.
\begin{Def}
A continuous real-valued function $g$ on a graph $G$ is said to be {\it good} if it is injective on each edge, which is homeomorphic to a closed interval.
\end{Def}
In the present paper, we first show the following result. The unit disk $D^k$, {\it Morse} functions. {\it Height} functions, {\it fold} maps, {\it special generic} maps, and so on, are explained in the next section.
We omit expositions on classical fundamental theory on (connected and) compact (closed) surfaces and their genera, Euler numbers and topologies.
\begin{MainThm}
\label{mthm:1}
Let $G$ be a finite and connected graph $G$ which has at least one edge and no loops. Let there exist a good function $g$ on $G$.
Suppose that an integer is assigned to each edge by an integer-valued function $r_G$ on the edge set satisfying either of the following two for each edge $e$. 
\begin{enumerate}
\item $r_G(e) \geq 0$.
\item $r_G(e)$ is even and negative.
\end{enumerate}
Then there exist a $3$-dimensional closed, connected and orientable manifold $M$ and a smooth function $f$ on $M$ satisfying the following four properties.
\begin{enumerate}
\item The Reeb graph $W_f$ of $f$ is isomorphic to $G${\rm :} we can take a suitable isomorphism $\phi:W_f \rightarrow G$ compatible with the remaining properties.
\item If we consider the natural quotient map $q_f:M \rightarrow W_f$ and for each point $\phi(p) \in G$ {\rm (}$p \in W_f${\rm )} that is not a vertex and that is in an edge $e$, then the preimage ${q_f}^{-1}(p)$ is a closed, connected, and orientable surface of genus $r_G(e)$ if $r_G(e) \geq 0$ and a non-orientable one whose genus is $-r_G(e)$ if $r_G(e) < 0$.
\item For a point $p \in M$ mapped by $q_f$ to a vertex $v_p:=q_f(p) \in W_f$, $f(p)=g \circ \phi(v_p)$.
\item Around each singular point $p$, the function has either of the following forms.
\begin{enumerate}
\item A Morse function if at the vertex $q_f(p)$ $g$ does not have a local extremum.
\item A height function if the vertex $q_f(p)$ is of degree $1$, at the vertex $q_f(p)$ $g$ has a local extremum and at the edge $e$ containing the vertex $r_G(e)=0$.
\item A composition of a Morse function with a height function if the vertex $q_f(p)$ is of degree greater than $1$ and $g$ has a local extremum there.
\item Let the vertex $q_f(p)$ be of degree $1$, at the vertex $q_f(p)$ $g$ have a local extremum and let us assume that at the edge $e$ containing the vertex $r_G(e)=-2$. In this case, the form is obtained in the following way.
\begin{enumerate}
\item Define a smooth function $h$ such that $h(x)=0$ for $x\leq 0$ and that $h(x)={\epsilon}_h e^{-\frac{1}{x}}$ for $x>0$ and a sufficiently small ${\epsilon}_h>0$.
\item Consider a special generic map on the total space of a non-trivial smooth bundle over a circle whose fiber is diffeomorphic to the $2$-dimensional unit sphere into the plane such that the restriction to the singular set is a smooth embedding and that the image of the singular set is the disjoint of two circles, which are centered at $(0,0)$ and whose radii are $\frac{1}{2}$ and $\frac{3}{2}$, respectively, where the underlying metric is the standard Euclidean metric, and that the image of the map is the closure of the domain surrounded by the two embedded circles. 
\item We consider the restriction of the previous map to the preimage of the unit disk $D^2$ in the plane and compose the previous function with a height function mapping $(x_1,x_2)$ to ${x_1}^2+{x_2}^2$.
\item We compose the resulting function with a smooth function mapping $x$ to $\pm h(x-\frac{1}{4})$ with the domain restricted suitably. 
\item We consider the sum of the previous function and a real-valued constant function. This is the form we consider here.
\end{enumerate}
Furthermore, the preimage of $q_f(p)$ is a circle.
\item Let the vertex $q_f(p)$ be of degree $1$, let $g$ have a local extremum there and let us assume that at the edge $e$ containing the vertex $r_G(e) \neq 0,-2$. In this case, the form is a composition of a fold map into the open unit disk with a height function. Furthermore, the fold map is constructed as a map satisfying the following two.
\begin{enumerate}
\item Each connected component of the preimage of each point in the open unit disk of target is a circle, a bouquet of two circles, or a 1-dimensional polyhedron obtained by an iteration of identifying two points whose small open neighborhoods are homeomorphic to an open interval in distinct connected $1$-dimensional polyhedra, starting from finitely many circles.
\item For the singular set of the fold map, remove finitely many singular points and consider the restriction there. This is an embedding. The finitely many singular points removed before are all in a same preimage or the preimage of the origin for the fold map.  
\end{enumerate}
\end{enumerate}
\end{enumerate}
\end{MainThm}
As Remark \ref{rem:1} says, once it was announced that we can immediately show this. However, we need additional arguments.  
As another result, we also show the following result. This is new including the situation.
\begin{MainThm}
\label{mthm:2}
Let $G$ be a finite and connected graph which has at least one edge and no loops. Let there exist a good function $g$ on $G$.
Suppose that an integer is assigned to each edge by an integer-valued function $r_G$ on the edge set satisfying the following two.
\begin{enumerate}
\item For each vertex $v$, the difference $D_v:=\sharp A_{{\rm up},v}-\sharp A_{{\rm low},v}$ of the two following numbers are even.
\begin{enumerate}
\item The number of edges containing $v$ as the point at which the restrictions of the function $g$ to the edges have the minima such that the values of the function $r_G$ there are odd and negative: let $A_{{\rm up},v}$ denote the set of all such edges and $\sharp A_{{\rm up},v} \geq 0$ denote the size of the set, which is an integer. 
\item The number of edges containing $v$ as the point at which the restrictions of the function $g$ to the edges have the maxima such that the values of the function $r_G$ there are odd and negative: let $A_{{\rm low},v}$ denote the set of all such edges and $\sharp A_{{\rm low},v} \geq 0$ denote the size of the set, which is an integer. 
\end{enumerate}
\item Let $v$ be an arbitrary vertex of $G$ at which $g$ does not have a local extremum and $D_v \neq 0$.
\begin{enumerate}
\item Let $D_v>0$. Let $B_{{\rm low},v}$ denote the set of all edges containing $v$ as the point at which the restrictions of the function $g$ to the edges have the maxima such that the values of the function $r_G$ there are negative. For each edge $e \in {B_{{\rm low},v}}$ define ${r_G}^{\prime}(e)$ as the greatest even number satisfying ${r_G}^{\prime}(e) \leq |r_G(e)|$. The sum ${\Sigma}_{e \in B_{{\rm low},v}} {r_G}^{\prime}(e)$ satisfies $D_v \leq {\Sigma}_{e \in B_{{\rm low},v}} {r_G}^{\prime}(e)$.
\item Let $D_v<0$. Let $B_{{\rm up},v}$ denote the set of all edges containing $v$ as the point at which the restrictions of the function $g$ to the edges have the minima such that the values of the function $r_G$ there are negative. For each edge $e \in {B_{{\rm up},v}}$ define ${r_G}^{\prime}(e)$ as the greatest even number satisfying ${r_G}^{\prime}(e) \leq |r_G(e)|$. The sum ${\Sigma}_{e \in B_{{\rm up},v}} {r_G}^{\prime}(e)$ satisfies $|D_v| \leq {\Sigma}_{e \in B_{{\rm up},v}} {r_G}^{\prime}(e)$.
\end{enumerate}
\end{enumerate}
Then there exist a $3$-dimensional closed and connected manifold $M$ and a smooth function $f$ on $M$ satisfying the four properties in Main Theorem \ref{mthm:1}.

\end{MainThm} 
We explain relationship among \cite{kitazawa}, \cite{saeki4} and the present paper.
\cite{kitazawa} concerns cases where the $3$-dimensional manifolds and preimages of regular values are orientable and Main Theorems extends this and this says that the author has a good idea for the present problem (Problem 3).  
\cite{saeki4} generalizes a main result of \cite{kitazawa} partially and regarded as a paper motivated by the paper. 
This also generalize Main Theorems partially. 

\cite{saeki4} generalizes the manifold assigned to each edge to a general compact (closed) manifold and for edges containing a vertex, the manifolds satisfy a condition from theory of cobordisms of manifolds.

On the other hand, explicit classes of smooth functions are not considered and used functions are essentially ones obtained by integrating so-called {\it bump} functions in \cite{saeki4}. A {\it bump} function is a kind of smooth
 maps which are not real-analytic. In \cite{kitazawa} and the present paper more explicit functions are used. 

Moreover, other related studies introduced before are essentially ones for smooth functions on surfaces or Morse functions such that preimages of regular values are disjoint unions of spheres.  

We prove Main Theorems in the next section.
\section{Several terminologies and the proof of Main Theorem \ref{mthm:1}}
\begin{Def}
\label{def:3}
A smooth map from a manifold of dimension $m>0$ with no boundary into a manifold of dimension $n>0$ with no boundary is said to be a {\it fold} map if the relation $m \geq n$ holds and at each singular point $p$, there exists an integer satisfying $0 \leq i(p) \leq \frac{m-n+1}{2}$ and the map has the form $(x_1,\cdots,x_m) \mapsto (x_1,\cdots,x_{n-1},\sum_{k=n}^{m-i(p)}{x_k}^2-\sum_{k=m-i(p)+1}^{m}{x_k}^2)$ for suitable coordinates. 
\end{Def}
The case where the manifold of the target is the line is for {\it Morse} functions on manifolds with no boundaries. We mainly consider Morse functions on smooth manifolds with non-empty boundaries or non-compact manifolds with no boundaries in the present paper.
\begin{Prop}
For a fold map in Definition \ref{def:3}, the integer $i(p)$ is unique for any singular point $p$ and the set of all singular points of an arbitrary fixed $i(p)$ is a closed and smooth submanifold of dimension $n-1$ with no boundary such that the map obtained by restricting the original map there is a smooth immersion.  
\end{Prop}
In this proposition, we call $i(p)$ the {\it index} of $p$.
\begin{Def}
A fold map is said to be {\it special generic} if the index of each singular point is $0$.
\end{Def}

For fundamental theory on singularity theory and differential topological properties of fold maps, see \cite{golubitskyguillemin}, \cite{saeki}, \cite{saeki2}, among the others.

We prove our Main Theorems.

Hereafter, the $k$-dimensional Euclidean space ${\mathbb{R}}^k$ including the line ($\mathbb{R}$) and the plane (${\mathbb{R}}^2$) are endowed with the standard Euclidean metrics.
The sphere which is centered at the origin of ${\mathbb{R}}^{k+1}$ and whose radius is $1$ is the $k$-dimensional unit sphere.
The (open) disk which is centered at the origin of ${\mathbb{R}}^{k}$ and whose radius is $1$ is the $k$-dimensional (resp. open) unit disc. $D^k$ denotes the unit disc, which is a closed subset. The open unit disk is its interior.

{\it Height} functions of (open) disks or (open) unit disks are Morse functions with exactly one singular point $p$ with $i(p)=0$ (in the interior if the disk is not open). In other words, height functions are functions having the form $(x_1,\cdots,x_m) \mapsto \pm {\Sigma}_{j=1}^m {x_j}^2+c$ for suitable coordinates and a constant $c \in \mathbb{R}$. 
\begin{proof}[A proof of Main Theorem \ref{mthm:1}.]
We need to add several arguments to the proof of Theorem 1 of \cite{kitazawa}.\\

First we introduce several smooth functions. Let $a<b$ be real numbers and $\{t_j\}_{j=1}^l$ be a sequence of real numbers in $(a,b)$ of length $l \geq 0$.
$\tilde{f}_{{\rm P},\{t_j\},[a,b]}$ denotes a function satisfying the following six. 
\begin{enumerate}
\item The manifold of the domain is represented as a connected sum of $l$ copies of the real projective plane with the interiors of two copies of the $2$-dimensional unit disk  smoothly and disjointly embedded removed.
\item The Reeb space is homeomorphic to a closed interval.
\item The preimage of $\{a,b\}$ and the boundary of the manifold of the domain agree.  
\item There exist exactly $l$ singular points. We can take the sequence $\{s_j\}_{j=1}^l$ of all singular points so that $\tilde{f}_{{\rm P},\{t_j\},[a,b]}(s_j)=t_j$ holds.
\item Around each singular point, the function is a Morse function.
\item Preimages of regular values are always circles.
\end{enumerate}

${\tilde{f}}_{{\rm K,u},t,[a,b]}$ (${\tilde{f}}_{{\rm K,d},t,[a,b]}$) denotes a function on a $3$-dimensional compact and connected manifold satisfying the following five.
\begin{enumerate}
\item The function has exactly one singular point and $t$ is the singular value. The function is a Morse function there.
\item The Reeb space is homeomorphic to a closed interval.
\item The preimage of $\{a,b\}$ and the boundary of the manifold of the domain agree.
\item The preimage of $a$ (resp. $b$) is diffeomorphic to the $2$-dimensional unit sphere.
\item The preimage of $b$ (resp. $a$) is diffeomorphic to the Klein bottle.
\end{enumerate}

We also need several fundamental arguments on relationship between Morse functions and singular points, and the manifolds. In fact the theory of so-called {\it handles} will help us to understand the structures of the presented functions and smooth functions which are in some senses similar introduced later. In the proof of Main Theorem \ref{mthm:2}, we explain the theory of handles in short with a short proof on the construction of another new function $\tilde{f}_{{\rm P},t,[a,b]}$.

For systematic theory on relationship between handles and singular points of Morse functions, see \cite{milnor} for example.

We come back to the proof. Besides the original proof of Theorem 1 of \cite{kitazawa}, we need the following four. \\
\ \\
Ingredient 1 Construction around a vertex $v \in G$ contained in some edge $e$ satisfying $r_G(e)<0$ where the function $g$ does not have a local extremum. \\
We consider a small regular neighborhood $N(v) \subset G$ and we can regard this as a graph whose edge set consists of closed intervals being subsets of edges of $G$ and containing $v$ and which satisfies the following two.

\begin{enumerate}
\item Mutually distinct edges of the graph $N(v)$ are always closed intervals in mutually distinct edges of $G$.
\item For each edge of $G$ containing $v$, there exists a unique edge of $N(v)$ being a closed interval in the edge of $G$.
\end{enumerate}
  
By methods of Michalak (\cite{michalak}) and the author (\cite{kitazawa}), we construct a local smooth function $\tilde{f_{v,0}}$ onto a small closed interval $[g(v)-{\epsilon}_v,g(v)+{\epsilon}_v]$ satisfying the following five properties where ${\epsilon}_v>0$ is a sufficiently small real number.
\begin{enumerate}
\item $\tilde{f_{v,0}}$ is a Morse function with exactly one value $g(v)$.
\item The preimage of $\{g(v)-{\epsilon}_v,g(v)+{\epsilon}_v\}$ and the boundary of the manifold of the domain agree.
\item There exists a suitable isomorphism ${\phi}_{v,0}$ from the Reeb space $W_{\tilde{f_{v,0}}}$ onto $N(v)$ mapping ${(\bar{\tilde{f_{v,0}}})}^{-1}(g(v)) \subset W_{{\tilde{f}}_{v,0}}$ to the one-point set $\{v\}$ compatible with the present four properties where $W_{\tilde{f_{v,0}}}$ has the structure of a graph such that the vertex set consists of the following two.
\begin{enumerate}
\item All elements in ${(\bar{\tilde{f_{v,0}}})}^{-1}(\{g(v)-{\epsilon}_v,g(v)+{\epsilon}_v\})$.
\item The unique element in ${(\bar{\tilde{f_{v,0}}})}^{-1}(g(v))$. 
\end{enumerate}  
\item For the edge $e^{\prime} \ni v$ of $N(v)$ contained in the edge $e$ of $G$ and each point $x \in e^{\prime}-\{v\}$, ${({\phi}_{v,0} \circ q_{\tilde{f_{v,0}}})}^{-1}(x)$ is a closed, connected and orientable surface of genus $r_G(e)$ if $r_G(e) \geq 0$.
\item For the edge $e^{\prime} \ni v$ of $N(v)$ contained in the edge $e$ of $G$ and each point $x \in e^{\prime}-\{v\}$, ${({\phi}_{v,0} \circ q_{\tilde{f_{v,0}}})}^{-1}(x)$ is diffeomorphic to the $2$-dimensional unit sphere if $r_G(e)<0$.
\end{enumerate}
We change this function to a desired local smooth function $\tilde{f_v}$. We can choose finitely many trivial smooth bundles over the image $[g(v)-{\epsilon}_v,g(v)+{\epsilon}_v]$ whose fibers are diffeomorphic to the unit disk $D^2$ disjointly and apart from the singular set of the function. 
We apply such arguments throughout the present paper and this is due to the structures of Morse functions with a relative version of an important theorem in \cite{ehresmann} saying that smooth submersions with compact preimages give smooth bundles.
We can attach new functions instead. We do this procedure according to the following rule.
\begin{enumerate}
\item For the edge $e^{\prime} \ni v$ of $N(v)$ contained in the edge $e$ of $G$ satisfying $r_G(e)<0$ and $\bar{\tilde{f_{v,0}}}(e^{\prime})=[g(v)-{\epsilon}_v,g(v)]$, first we choose $\frac{|r_G(e)|}{2}$ trivial smooth bundles over the image $[g(v)-{\epsilon}_v,g(v)+{\epsilon}_v]$ whose fibers are diffeomorphic to the unit disk $D^2$ disjointly and apart from the singular set of the function so that the images by the restriction of $q_{\tilde{f_{v,0}}}$ to ${\tilde{f_{v,0}}}^{-1}([g(v)-{\epsilon}_v,g(v)])$ are all $e^{\prime}$. We attach copies of the function ${\tilde{f}}_{{\rm K,d},g(v),[g(v)-{\epsilon}_v,g(v)+{\epsilon}_v]}$ before with a trivial smooth bundle over $[g(v)-{\epsilon}_v,g(v)+{\epsilon}_v]$ whose fiber is diffeomorphic to the unit disk $D^2$ and which is apart from the singular set removed instead (we can construct this function by a fundamental argument on the structure of the function ${\tilde{f}}_{{\rm K,d},g(v),[g(v)-{\epsilon}_v,g(v)+{\epsilon}_v]}$). In the attachment, we must preserve the values at all points of the manifolds of the domains.
\item For the edge $e^{\prime} \ni v$ of $N(v)$ contained in the edge $e$ of $G$ satisfying $r_G(e)<0$ and $\bar{\tilde{f_{v,0}}}(e^{\prime})=[g(v),g(v)+{\epsilon}_v]$, first we choose $\frac{|r_G(e)|}{2}$ trivial smooth bundles over the image $[g(v)-{\epsilon}_v,g(v)+{\epsilon}_v]$ whose fibers are diffeomorphic to the unit disk $D^2$ disjointly and apart from the singular set of the function so that the images by the restriction of $q_{\tilde{f_{v,0}}}$ to ${\tilde{f_{v,0}}}^{-1}([g(v),g(v)+{\epsilon}_v])$ are all $e^{\prime}$. We attach copies of the function ${\tilde{f}}_{{\rm K,u},g(v),[g(v)-{\epsilon}_v,g(v)+{\epsilon}_v]}$ before with a trivial smooth bundle over $[g(v)-{\epsilon}_v,g(v)+{\epsilon}_v]$ whose fiber is diffeomorphic to the unit disk $D^2$ and which is apart from the singular set removed instead (we can construct this function by a fundamental argument on the structure of the function ${\tilde{f}}_{{\rm K,u},g(v),[g(v)-{\epsilon}_v,g(v)+{\epsilon}_v]}$). In the attachment, we must preserve the values at all points of the manifolds of the domains.
\end{enumerate}
Thus we have a desired local smooth function onto $[g(v)-{\epsilon}_v,g(v)+{\epsilon}_v]$.\\
\ \\
Hereafter, we omit rigorous notation and expositions on identifications of original abstract graphs and Reeb graphs of local or global smooth functions. We naturally identify them in similar arguments.  \\
\ \\
Ingredient 2 Construction around a vertex $v$ of degree greater than $1$ contained in some edge $e$ satisfying $r_G(e)<0$ where the function $g$ has a local extremum. \\
As in \cite{kitazawa}, we construct a local function as Ingredient $1$ and compose this with a smooth embedding into the plane whose image is a parabola $\{(t,g(v) \pm t^2) \mid t \in [-{\epsilon}_v,{\epsilon}_v].\} \subset {\mathbb{R}}^2$ so that the singular point of the local function is mapped to $(0,g(v))$ where ${\epsilon}_v>0$ is a sufficiently small real number. We compose the resulting map into the plane with the canonical projection to the second component and we have a desired function. We choose the sign $+$ ($-$) according to the condition that $g(v)$ is the local minimum (resp. maximum) of $g$. \\
\ \\ 
Ingredient 3 Construction around a vertex $v$ of degree $1$ contained in the edge $e$ satisfying $r_G(e)=-2$ where the function $g$ has a local extremum. \\
\ \\
We can have a special generic map on the total space of a non-trivial smooth bundle over a circle whose fiber is diffeomorphic to the $2$-dimensional unit sphere into the plane such that the restriction to the singular set is a smooth embedding. We can obtain the map so that the image of the singular set is the disjoint of two circles, which are centered at $(0,0)$ and whose radii are $\frac{1}{2}$ and $\frac{3}{2}$, respectively, and that the image of the map is the closure of the domain surrounded by the two embedded circles. We consider the restriction of the map to the preimage of the unit disk $D^2$ in the plane and compose this with a height function mapping $(x_1,x_2)$ to ${x_1}^2+{x_2}^2$. We compose this with a smooth function mapping $x$ to $\pm h(x-\frac{1}{4})$ with the domain restricted suitably. Here the sign $+$ ($-$) is chosen according to the condition that $g(v)$ is the local minimum (resp. maximum) of $g$. We consider the sum of the resulting function and the constant function whose values are always $g(v)$ and we thus have a desired local function. \\
\ \\
Ingredient 4 Construction around a vertex $v$ of degree $1$ contained in the edge $e$ satisfying $r_G(e) \neq -2$ and $r_G(e)=-2(l_0+1)$ for a positive integer $l_0>0$ where the function $g$ has a local extremum. \\
\ \\
Let ${\epsilon}_v>0$ be a small real number. We consider a smooth deformation $F_v:P \times [-1,1] \rightarrow \mathbb{R} \times [-1,1]$ where $P$ is a surface of the domain of ${\tilde{f}}_{{\rm P},\{t_j\},[-{\epsilon}_v,{\epsilon}_v]}$ such
that the length of $\{t_j\}$ is $l_0$. 
Let $F_{v,t}:P \rightarrow \mathbb{R}$ denote the map defined uniquely by the relation $F_v(x,t)=(F_{v,t}(x),t)$. Set $t_j$ in the following way.
\begin{itemize}
\item For $l_0=1$, let $t_j=t_1:=0$.
\item For $l_0>1$, let $t_1:=-\frac{{\epsilon}_v}{2}$ and $t_{l_0}:=\frac{{\epsilon}_v}{2}$ and $t_j:=-\frac{{\epsilon}_v}{2}+(j-1)\frac{{\epsilon}_v}{l_0-1}$ for $1<j<l_0$.
\end{itemize}
We can define ${\tilde{f}}_{{\rm P},\{t_j+\frac{t+1}{2}(t_{l_0-j+1}-t_j)\},[-{\epsilon}_v,{\epsilon}_v]}$ for $-1 \leq t \leq 1$ and set $F_{v,t}:={\tilde{f}}_{{\rm M},\{t_j+\frac{t+1}{2}(t_{l_0-j+1}-t_j)\},[-{\epsilon}_v,{\epsilon}_v]}$ for $-1 \leq t \leq 1$ suitably to define a smooth deformation $F_v:P \times [-1,1] \rightarrow \mathbb{R} \times [-1,1]$ such that the following properties hold.
\begin{enumerate}
\item The singular set consists of all points of the form $(s_{t,j},t)$ where $s_{t,j}$ is a singular point of the Morse function $F_{v,t}$.
\item On the interior of $P \times [-1,1]$, represented as ${\rm Int}\ P \times (-1,1)$, the map $F_v$ is a fold map.
\item For the fold map before, the restriction to the subset of the singular set obtained by removing all singular points in the preimage of $(0,0)$ is a smooth embedding. 
\end{enumerate}

We compose the restriction of this map to the preimage of the open disk which is centered at $(0,0)$ and whose radius is $\frac{1}{2}$ with a height function mapping $(x_1,x_2)$ to $\pm({x_1}^2+{x_2}^2)$.
We consider the sum of the resulting function and a constant function whose values are always $g(v)$ and we have a desired local function.

Note that for a sufficiently small real number $x_{g(v)}>0$ the preimage of $g(v) \pm x_{g(v)}$ is diffeomorphic to a closed, connected and non-orientable surface obtained by identifying the two connected components of the boundary of the surface of the domain of $\tilde{f}_{{\rm P},\{t_j\},[a,b]}$ with $l:=2l_0$. The genus is $2+2l_0$. Note also that the sign $+$ or $-$ is chosen according to the condition that $g(v)$ is the local minimum (resp. maximum) of $g$.

For this ingredient, see also \cite{saeki3}, explaining the differential topological theory on structures around preimages of smooth maps whose codimensions are $-1$: so-called theory of {\it fibers}. \\
\ \\
\ \\
Except these four, we can show as Theorem 1 of \cite{kitazawa}. This completes the proof.
\end{proof}
\begin{Rem}
\label{rem:1}
In an earlier version of \cite{kitazawa}, it was announced that we immediately have a similar answer to Problem 3 of \cite{kitazawa}. However, the answer is not so trivial since we need Ingredient 3 for example. 
\end{Rem}
We prepare two important lemmas.
\begin{Lem}
\label{lem:1}
Let $m>1$ be an integer.
Suppose two smooth functions $\tilde{f_1}$ and $\tilde{f_2}$ on an $m$-dimensional compact, connected and smooth manifold onto a small closed interval $[a-{\epsilon}_a,a+{\epsilon}_a]$ satisfying the following three properties where ${\epsilon}_a>0$ is a real number.
\begin{enumerate}
\item $\tilde{f_1}$ and $\tilde{f_2}$ are both Morse functions with exactly one singular value $a$.
\item The preimages of $a-{\epsilon}_a$ are diffeomorphic to closed manifolds $F_{{\rm low},1}$ and $F_{{\rm low},2}$ respectively. The preimages of $a+{\epsilon}_a$ are diffeomorphic to closed manifolds $F_{{\rm up},1}$ and $F_{{\rm up},2}$ respectively. 
\item The preimages ${\bar{\tilde{f_1}}}^{-1}(a)$ and ${\bar{\tilde{f_2}}}^{-1}(a)$ are both one-point sets.
\end{enumerate}
Then we have a smooth function $\tilde{f_{1,2}}$ on an $m$-dimensional compact, connected and smooth manifold onto $[a-{\epsilon}_a,a+{\epsilon}_a]$ satisfying the following five properties.
\begin{enumerate}
\item $\tilde{f_{1,2}}$ is a Morse function with exactly one singular value $a$.
\item The preimage of $a-{\epsilon}_a$ is diffeomorphic to the disjoint union of $F_{{\rm low},1}$ and $F_{{\rm low},2}$. The preimage of $a+{\epsilon}_a$ is diffeomorphic to the disjoint union of $F_{{\rm up},1}$ and $F_{{\rm up},2}$.
\item The preimage ${\bar{\tilde{f_{1,2}}}}^{-1}(a)$ is a one-point set.
\end{enumerate}
\end{Lem}
\begin{proof}
Besides $\tilde{f_1}$ and $\tilde{f_2}$, we construct a similar smooth function $\tilde{f_{\rm S}}$ on $m$-dimensional compact, connected and smooth manifolds onto a small closed interval $[a-{\epsilon}_a,a+{\epsilon}_a]$ satisfying the following three as in Ingredient 1 of the proof of Main Theorem \ref{mthm:1}.
\begin{enumerate}
\item This is a Morse function with exactly one singular value $a$.
\item The preimages of $a-{\epsilon}_a$ and $a+{\epsilon}_a$ are diffeomorphic to a disjoint union of two copies of the ($m-1$)-dimensional unit disk.
\item The preimage $\tilde{f_{\rm S}}(a)$ is a one-point set.
\end{enumerate}
For $\tilde{f_{\rm S}}$, we can choose two trivial smooth bundles over the image $[a-{\epsilon}_a,a+{\epsilon}_a]$ whose fibers are diffeomorphic to the unit disk $D^{m-1}$ disjointly and apart from the singular set of the function as in Ingredient 1 of the proof of Main Theorem \ref{mthm:1}. We can also choose them so that the total spaces are mapped to the unions of two edges by the quotient map to the Reeb space and that the sets of the two edges are disjoint for these two bundles: note that the Reeb space is regarded as a graph with exactly four edges.   
For $\tilde{f_j}$ ($j=1,2$) we can choose a trivial smooth bundle over the image $[a-{\epsilon}_a,a+{\epsilon}_a]$ whose fiber is diffeomorphic to the unit disk $D^{m-1}$ apart from the singular set of the function similarly.

We glue these functions together to obtain a desired smooth function as in Ingredient 1 of the proof of Main Theorem \ref{mthm:1}. This completes the proof.
\end{proof}
\begin{Lem}
\label{lem:2}
Let $m>1$ be an integer.
Suppose two smooth functions $\tilde{f_1}$ and $\tilde{f_2}$ on an $m$-dimensional compact, connected and smooth manifold onto a small closed interval $[a-{\epsilon}_a,a+{\epsilon}_a]$ satisfying the following three properties where ${\epsilon}_a>0$ is a real number.
\begin{enumerate}
\item $\tilde{f_1}$ and $\tilde{f_2}$ are both Morse functions with exactly one singular value $a$.
\item The preimages of $a-{\epsilon}_a$ are diffeomorphic to closed manifolds $F_{{\rm low},1}$ and $F_{{\rm low},2}$ respectively. 
Let $F_{{\rm low},1,0} \subset F_{{\rm low},1}$ and $F_{{\rm low},2,0} \subset F_{{\rm low},2}$ be connected components of the manifolds.
The preimages of $a+{\epsilon}_a$ are diffeomorphic to closed manifolds $F_{{\rm up},1}$ and $F_{{\rm up},2}$ respectively. 
\item The preimages ${\bar{\tilde{f_1}}}^{-1}(a)$ and ${\bar{\tilde{f_2}}}^{-1}(a)$ are both one-point sets in the Reeb spaces.
\end{enumerate}
Then we have a smooth function $\tilde{f_{1,2}}$ on an $m$-dimensional compact, connected and smooth manifold onto $[a-{\epsilon}_a,a+{\epsilon}_a]$ satisfying the following five properties.
\begin{enumerate}
\item $\tilde{f_{1,2}}$ is a Morse function with exactly one singular value $a$.
\item The preimage of $a-{\epsilon}_a$ is diffeomorphic to the disjoint union of a manifold diffeomorphic to $F_{{\rm low},1}-F_{{\rm low},1,0}$, a manifold diffeomorphic to $F_{{\rm low},2}-F_{{\rm low},2,0}$, and a manifold represented as a connected sum of the two manifolds $F_{{\rm low},1,0}$ and $F_{{\rm low},2,0}$ considered in the smooth category. The preimage of $a+{\epsilon}$ is diffeomorphic to the disjoint union of $F_{{\rm up},1}$ and $F_{{\rm up},2}$.
\item The preimage ${\bar{\tilde{f_{1,2}}}}^{-1}(a)$ is a one-point set.
\end{enumerate}
\end{Lem}
\begin{proof}
Besides $\tilde{f_1}$ and $\tilde{f_2}$, we construct a similar smooth function $\tilde{f_{\rm S,up}}$ on $m$-dimensional compact, connected and smooth manifolds onto a small closed interval $[a-{\epsilon}_a,a+{\epsilon}_a]$ satisfying the following three as in Ingredient 1 of the proof of Main Theorem \ref{mthm:1}.
\begin{enumerate}
\item This is a Morse function with exactly one singular value $a$.
\item The preimages of $a-{\epsilon}$ and $a+{\epsilon}$ are diffeomorphic to the ($m-1$)-dimensional unit disk and a disjoint union of two copies of the ($m-1$)-dimensional unit disk respectively. 
\item The preimage $\tilde{f_{\rm S,up}}(a)$ is a one-point set.
\end{enumerate}
For $\tilde{f_{\rm S,up}}$, we can choose two trivial smooth bundles over the image $[a-{\epsilon}_a,a+{\epsilon}_a]$ whose fibers are diffeomorphic to the unit disk $D^{m-1}$ disjointly and apart from the singular set of the function as in Ingredient 1 of the proof of Main Theorem \ref{mthm:1}. We can also choose them so that the total spaces are mapped to the unions of two edges by the quotient map to the Reeb space and that the pairs of the two edges are distinct for these two bundles: note that the Reeb space is regarded as a graph with exactly three edges. 

We identify $F_{{\rm low},1,0} \subset F_{{\rm low},1}$ and $F_{{\rm low},2,0} \subset F_{{\rm low},2}$ with the preimages of $a-{\epsilon}$ for functions $\tilde{f_1}$ and  $\tilde{f_2}$ in a suitable way. 
We can choose a trivial smooth bundle over the image $[a-{\epsilon}_a,a+{\epsilon}_a]$ whose fiber is diffeomorphic to the unit disk $D^{m-1}$ and a fiber of which is in $F_{{\rm low},j,0} \subset F_{{\rm low},j}$ apart from the singular set of each of these two functions as in Ingredient 1 of the proof of Main Theorem \ref{mthm:1}.
We glue these functions together to obtain a desired smooth function as in Ingredient 1 of the proof of Main Theorem \ref{mthm:1}. This completes the proof.
\end{proof}
\begin{proof}[A proof of Main Theorem \ref{mthm:2}.]

We first define three kinds of smooth functions.
Let $a<t<b$ be real numbers and let $l>0$. Let $\tilde{f}_{{\rm P,0},t,[a,b]}$ denote a smooth function on a $3$-dimensional compact and connected manifold satisfying the following five.
\begin{enumerate}
\item The Reeb space is homeomorphic to a closed interval.
\item The preimage of $\{a,b\}$ and the boundary of the manifold of the domain agree.  
\item The preimages of $a$ and $b$ are both diffeomorphic to the real projective plane.
\item There exist exactly two singular points. $t$ is the unique singular value.
\item Around each singular point, the function is a Morse function.
\end{enumerate}

We explain the structure of such a function, since this needs a non-trivial argument. We consider a product of a copy of the real projective plane $P_S$ and a closed interval $[0,1]$. We attach a so-called {\it $1$-handle}, diffeomorphic to $D^2 \times [0,1] \supset D^2 \times \{0,1\}$, to a smoothly and disjointly embedded two copies of the unit disk $D^2$ in $P_S \times \{0\}$. We also attach a so-called {\it $2$-handle}, diffeomorphic to $D^2 \times [0,1] \supset \partial D^2 \times [0,1]$, to a smoothly embedded surface, diffeomorphic to $\partial D^2 \times [0,1]$ and apart from the two embedded copies of the disk $D^2$ before, in $P_S \times \{0\}$. 
Furthermore, we can do this so that after removing the surface diffeomorphic to $\partial D^2 \times [0,1]$ from $P_S \times \{0\}$, remaining two connected components are diffeomorphic to the $2$-dimensional open unit disk and the interior of the M\"{o}bius band respectively and that the two embedded copies of the disk $D^2$ before are in different connected components. This is a non-trivial argument on this construction and this produces a desired function. This theory is used in the proof of main theorems and propositions of \cite{michalak} and (an earlier version of) \cite{kitazawa} explains the main theorem using this.

For the well-known fundamental correspondence between singular points of Morse functions and so-called ({\it $k$-}){\it handles} where $k$ is a non-negative integer smaller than or equal to the dimension of the manifold of the domain, consult \cite{milnor} as presented in the proof of Main Theorem \ref{mthm:1}.  

${\tilde{f}}_{{\rm N,up},l,t,[a,b]}$ (${\tilde{f}}_{{\rm N,low},l,t,[a,b]}$) denotes a function on a $3$-dimensional compact and connected manifold satisfying the following six. We can understand this by a fundamental argument on $1$-handles and we omit the exposition. 
\begin{enumerate}
\item The Reeb space is homeomorphic to a finite and connected graph.
\item The preimage of $\{a,b\}$ and the boundary of the manifold of the domain agree.
\item The preimage of $a$ (resp. $b$) is diffeomorphic to a disjoint union of $l+1$ copies of the real projective plane.
\item The preimage of $b$ (resp. $a$) is diffeomorphic to a closed, connected, and non-orientable surface of genus $1+l$.
\item There exist exactly $l$ singular points. $t$ is the unique singular value.
\item Around each singular point, the function is a Morse function.
\end{enumerate}

Let $a<t<b$ be real numbers.
${\tilde{f}}_{{\rm T,up},t,[a,b]}$ (${\tilde{f}}_{{\rm T,low},t,[a,b]}$) denotes a function on a $3$-dimensional compact and connected manifold satisfying the following five.
\begin{enumerate}
\item The function has exactly one singular point and $t$ is the singular value. The function is a Morse function there.
\item The Reeb space is homeomorphic to a closed interval.
\item The preimage of $\{a,b\}$ and the boundary of the manifold of the domain agree.
\item The preimage of $a$ (resp. $b$) is diffeomorphic to the $2$-dimensional unit sphere.
\item The preimage of $b$ (resp. $a$) is diffeomorphic to the torus.
\end{enumerate}
This is also in \cite{kitazawa}.

Proving that we can construct desired local functions in the following cases completes the proof.
\begin{itemize}
\item Around a vertex $v$ contained in some edge $e \in A_{{\rm low},v}$ or $e \in A_{{\rm up},v}$ where the function $g$ does not have a local extremum. 
\item Around a vertex $v$ contained in some edge $e \in A_{{\rm low},v}$ or $e \in A_{{\rm up},v}$ where the function $g$ has a local extremum. 
\end{itemize}
Part 1 The former case. \\
\indent We present the construction in the first case. First take a sufficiently small real number ${\epsilon}_v>0$. Hereafter, let $E_v$ denote the set of all edges containing $v$.
Let $E_{{\rm low},v}$ denote the set of all edges containing $v$ as the point at which the restrictions of the function $g$ to the edges have the maxima.
Let $E_{{\rm up},v}$ denote the set of all edges containing $v$ as the point at which the restrictions of the function $g$ to the edges have the maxima.\\

\noindent Case 1 The case where $\sharp A_{{\rm up},v}=\sharp A_{{\rm low},v}$.\\
\indent First prepare $\sharp A_{{\rm up},v}=\sharp A_{{\rm low},v}$ copies of the function $\tilde{f}_{{\rm P,0},g(v),[g(v)-{\epsilon}_v,g(v)+{\epsilon}_v]}$ before.

Let $E_{{\rm low},v}-A_{{\rm low},v}$ and $E_{{\rm up},v}-A_{{\rm up},v}$ be both empty. We apply Lemma \ref{lem:1} one after another and apply Ingredient 1 of the proof of Main Theorem \ref{mthm:1} to complete the proof. Note that the latter
 argument increases genera of closed and connected surfaces of preimages of points in the interiors of edges by $2$ where we regard the Reeb spaces naturally as graphs.

Let $E_{{\rm low},v}-A_{{\rm low},v}$ and $E_{{\rm up},v}-A_{{\rm up},v}$ be both non-empty. 
We first construct a desired function as the previous case where $A_{{\rm up},v}$ and $A_{{\rm low},v}$ play roles $E_{{\rm low},v}$ and $E_{{\rm up},v}$ play in the previous case.
We argue the case where $E_{{\rm low},v}-A_{{\rm low},v}$ and $E_{{\rm up},v}-A_{{\rm up},v}$ play roles $E_{{\rm low},v}$ and $E_{{\rm up},v}$ play in the case of Main Theorem \ref{mthm:1}: we restrict the original $r_G$ suitably and apply the proof of Main Theorem \ref{mthm:1}. We then apply Lemma \ref{lem:1} to complete the proof.

Let either $E_{{\rm low},v}-A_{{\rm low},v}$ or $E_{{\rm up},v}-A_{{\rm up},v}$ be non-empty. 
We first construct a desired function as the previous case where $A_{{\rm up},v}$ and $A_{{\rm low},v}$ play roles $E_{{\rm low},v}$ and $E_{{\rm up},v}$ play in the case in the beginning.
We choose the empty set between $E_{{\rm low},v}-A_{{\rm low},v}$ and $E_{{\rm up},v}-A_{{\rm up},v}$ and replace this by a one-element set.
We argue the case where these two sets play roles $E_{{\rm low},v}$ and $E_{{\rm up},v}$ play in the case of Main Theorem \ref{mthm:1}: we restrict the original $r_G$ suitably, define the value at the element of the one-element set before as $0$ and apply the proof of Main Theorem \ref{mthm:1} as before. We then apply Lemma \ref{lem:2} to complete the proof.\\
\ \\
This completes the construction of a desired function. \\
\ \\
\noindent Case 2 The case where $\sharp A_{{\rm up},v} \neq \sharp A_{{\rm low},v}$.\\
\indent Without loss of genericity, we may assume $\sharp A_{{\rm up},v} > \sharp A_{{\rm low},v}$. We can show similarly if $\sharp A_{{\rm up},v} < \sharp A_{{\rm low},v}$.

By the assumption on $D_v$ and ${r_G}^{\prime}(e)$ for $e \in {B_{{\rm low},v}}$, we can choose some edges of ${B_{{\rm low},v}}$, the set of all of which is denoted by ${C_{{\rm low},v}}$, and choose an even integer ${r_G}^{\prime \prime}(e)>0$ for $e \in {C_{{\rm low},v}}$ satisfying the following rules.
\begin{enumerate}
\item ${\Sigma}_{e \in {C_{{\rm low},v}}} {r_G}^{\prime \prime}(e)=D_v$.
\item ${r_G}^{\prime \prime}(e) \leq {r_G}^{\prime}(e)$.
\end{enumerate}
We prepare a copy of the function ${\tilde{f}}_{{\rm N,low},{r_G}^{\prime \prime}(e)-1,g(v),[g(v)-{\epsilon}_v,g(v)+{\epsilon}_v]}$ before for $e \in {C_{{\rm low},v}}$.

We construct a smooth function under the following conditions. 
\begin{enumerate}
\item We change the values of $r_G$ on $A_{{\rm up},v} \sqcup A_{{\rm low},v}$ to $-1$ and those on $(B_{{\rm up},v} \sqcup B_{{\rm low},v})-(A_{{\rm up},v} \sqcup A_{{\rm low},v})$ to $0$  . 
\item We remove $D_v$ edges in $A_{{\rm up},v}$ from $E_{{\rm up},v}$, without changing $E_{{\rm low},v}$.
\end{enumerate}
We construct a desired function as in Case 1.

We apply Lemma \ref{lem:2} one after another. At each step, set the preimage of $g(v)-{\epsilon}_v$ for the copy of the function ${\tilde{f}}_{{\rm N,low},{r_G}^{\prime \prime}(e)-1,g(v),[g(v)-{\epsilon}_v,g(v)+{\epsilon}_v]}$ before for $e \in {C_{{\rm low},v}}$ as the connected component identified with $F_{{\rm low},1,0}$ and set the intersection of the preimage of $g(v)-{\epsilon}_v$ and the preimage of $e$ for the latter function and the quotient map onto the Reeb space as the connected component identified with $F_{{\rm low},2,0}$.

Last we apply arguments in Ingredient 1 of the proof of Main Theorem \ref{mthm:1} to increase the genus of the preimage for each edge $e$ by $2k_e$ for a suitable integer $k_e \geq 0$. This completes the proof.  

\ \\
Part 2 The latter case. \\
\indent We present the construction in the latter case. 
We apply technique in Ingredient 2 of the proof of Main Theorem \ref{mthm:1}. We first construct a local function as in Part 1. We can and must construct the map regarding $\sharp A_{{\rm up},v}=\sharp A_{{\rm low},v}$ there by the conditions on $r_G$ and we can apply the presented technique to complete the construction.\\
\ \\
This completes the proof.

\end{proof}

\begin{Ex}
\label{ex:1}
We consider an explicit situation of Main Theorem \ref{mthm:2} where for a vertex $v$ the following three hold. We choose a sufficiently small real number ${\epsilon}_v>0$.

\begin{enumerate}
\item $E_v$ consists of exactly $5$ edges.
\item $E_{{\rm up},v}$ consists of exactly $3$ edges. The values of $r_G$ there are, $-1$, $-1$, and some non-negative integer, respectively.
\item At one edge of $E_{{\rm low},v}$, the value of $r_G$ is $-2$ and at the remaining edge, the value is some non-negative integer.
\end{enumerate}
\end{Ex}
We end the present paper by the following remarks.
\begin{Rem}
\label{rem:2}
In obtaining results similar to Main Theorems, for example, the first constraint that $D_v$ is even in Main Theorem \ref{mthm:2} is a necessary condition. This is due to constraints from the cobordism theory of closed manifolds. Consult \cite{saeki4}. 

There remain several cases satisfying the condition on this theory. For example, consider the following case for a vertex $v$.
\begin{enumerate}
\item $E_v$ consists of exactly $3$ edges.
\item $E_{{\rm up},v}$ consists of exactly $2$ edges. The values of $r_G$ there are odd and negative.
\item At the unique edge of $E_{{\rm low},v}$, the value of $r_G$ is non-negative. 
\end{enumerate}
Note that in the situation of \cite{saeki4}, we have a positive result. On the other hand, we do not use explicit functions such as Morse(-Bott) functions there. 
By fundamental discussions on handles we can see that we cannot construct Morse functions.
\end{Rem}
\begin{Rem}
In Main Theorems, functions we have obtained are Morse functions around each singular point $p$ at the vertex $q_f(p)$ where $g$ does not have a local extremum. We do not know whether we can weaken the conditions on $r_G$ or preimages of regular values so that this also holds.
\end{Rem}
\section{Acknowledgement.}
\label{sec:3}
\thanks{The author is a member of the project JSPS KAKENHI Grant Number JP17H06128 "Innovative research of geometric topology and singularities of differentiable mappings" (Principal Investigator: Osamu Saeki) and this work is supported by the project.

\end{document}